\documentclass[11pt]{amsart}
\usepackage{palatino}
\usepackage{amsfonts}
\usepackage{amssymb}
\usepackage{amscd}
\usepackage{xypic}
\usepackage[breaklinks,bookmarksopen,bookmarksnumbered]{hyperref}

\newtheorem{theorem}{Theorem}[section]
\newtheorem{proposition}[theorem]{Proposition}

\theoremstyle{definition}

\newtheorem{conjecture/question}[theorem]{Conjecture/Question}

\newtheorem{remark/definition}[theorem]{Remark/Definition}
\newtheorem{terminology/notation}[theorem]{Terminology/Notation}

\def\PP{{\textbf P}}
\def\OO{\mathcal{O}}

\def\cD{\mathcal{D}}

\def\cA{\mathcal{A}}
\def\F{\mathcal{F}}

\def\E{\mathcal{E}}
\def\G{\mathcal{G}}

\def\L{\mathcal{L}}
\def\I{\mathcal{I}}

\def\cM{\mathcal{M}}
\def\cV{\mathcal{V}}

\def\rr{\overline{\mathcal{R}}}

\def\cC{\mathcal{C}}
\def\H{\mathcal{H}}
\def\Pic0{{\rm Pic}^0(X)}

\def\mm{\overline{\mathcal{M}}}
\def\hh{\overline{\mathcal{H}}}

\begin{document}
\title{Effective divisors on Hurwitz spaces}
\author[G. Farkas]{Gavril Farkas}

\address{Humboldt-Universit\"at zu Berlin, Institut f\"ur Mathematik,  Unter den Linden 6
\hfill \newline\texttt{}
 \indent 10099 Berlin, Germany} \email{{\tt farkas@math.hu-berlin.de}}

\maketitle

\begin{center}
\emph{To Bill Fulton on the occasion of his 80th birthday.}
\end{center}

\begin{abstract}
We prove the effectiveness of the canonical bundle of a partial compactification of several Hurwitz spaces $\hh_{g,k}$ of degree $k$ admissible covers of $\PP^1$ from  curves of genus $14\leq g\leq 19$. This suggests that these moduli spaces are of general type.
\end{abstract}

\vskip 5pt

Following a principle due to Mumford, most moduli spaces that appear in algebraic geometry (classifying curves, abelian varieties, $K3$ surfaces) are of general type, with a finite number of exceptions, which are unirational, or at least uniruled. Understanding the transition  from negative Kodaira dimension to being of general type is usually quite difficult. The aim of this paper is to investigate to some extent this principle in the case of a prominent parameter space of curves, namely the Hurwitz space $\H_{g,k}$ classifying degree $k$ covers $C\rightarrow \PP^1$ with only simple ramification from a smooth curve $C$ of genus $g$. Hurwitz spaces provide an interesting bridge between the much more accessible moduli spaces of pointed rational curves and the moduli space of curves. The study of covers of the projective line emerges directly from Riemann's Existence Theorem and Hurwitz spaces viewed as classifying spaces for covers of curves have been classically used by Clebsch \cite{Cl} and Hurwitz \cite{Hu} to establish the irreducibility of the moduli space of curves. Fulton was the first to pick up the problem in the modern age and his influential paper \cite{Ful} initiated the treatment of Hurwitz spaces in arbitrary characteristic. Harris and Mumford \cite{HM} compactified Hurwitz spaces in order to show that $\mm_g$ is of general type for large $g$, see also \cite{H}. Their work led the way to the theory of admissible covers which has many current ramifications in Gromov-Witten theory.  Later,  Harris and Morrison \cite{HMo} bound from below the slope of all effective divisors on $\mm_g$ using the geometry of compactified Hurwitz spaces.

\vskip 4pt

We denote by $\hh_{g,k}$ the moduli space of admissible covers constructed by Harris and Mumford \cite{HM} and studied further in  \cite{ACV}, or in the book \cite{BR}. It comes equipped with a finite \emph{branch map} $$\mathfrak{b}\colon \hh_{g,k}\rightarrow \mm_{0, 2g+2k-2}/\mathfrak{S}_{2g+2k-2},$$ where the target is the moduli space  of \emph{unordered} $(2g+2k-2)$-pointed stable rational curves, as well as with a map
$$\sigma\colon \hh_{g,k}\rightarrow \mm_g,$$ obtained by assigning to an admissible cover the stable model of its source. The effective cone of divisors of the quotient $\mm_{0, 2g+2k-2}/\mathfrak{S}_{2g+2k-2}$ is spanned by the divisor classes $\widetilde{B}_2, \ldots, \widetilde{B}_{g+k-1}$, where $\widetilde{B}_i$ denotes the closure of the locus of stable rational curves consisting of two components such that one of them contains precisely $i$ of the $2g+2k-2$ marked points.

\vskip 4pt

It is a fundamental question to describe the birational nature of $\hh_{g,k}$ and one knows much less  than in the case of other moduli spaces like $\mm_g$ or $\overline{\cA}_g$. Recall that in a series of landmark papers \cite{HM}, \cite{H}, \cite{EH2} published in the 1980s, Harris, Mumford and Eisenbud proved that $\mm_g$ is a variety of general type for $g>23$. This contrasts with the classical result of Severi \cite{Sev} that $\mm_g$ is unirational for $g\leq 10$. Later it has been shown that the Kodaira dimension of $\mm_g$ is negative for $g\leq 15$. Precisely, $\mm_g$ is unirational for $g\leq 14$, see \cite{CR1}, \cite{Ser},  \cite{Ve}, whereas $\mm_{15}$ is rationally connected \cite{BV}. It has been recently established in \cite{FJP} that both $\mm_{22}$ and $\mm_{23}$ are of general type.

\vskip 4pt

From Brill-Noether theory it follows that when $\frac{g+2}{2}\leq k\leq g$, every curve of genus $g$ can be represented as a $k$-sheeted cover of $\PP^1$, that is, the map $\sigma\colon \hh_{g,k}\rightarrow \mm_g$ is dominant. Using \cite{HT} it is easy to see that the generic fibre of the map $\sigma$ is of general type, thus in this range $\hh_{g,k}$ is of general type whenever $\mm_g$ is. When $k\geq g+1$, the Hurwitz space is obviously uniruled being a Grassmannian bundle. Classically it has been known that $\H_{g,k}$ is unirational for $k\leq 5$, see \cite{AC} and references therein for a modern treatment (or \cite{Pe} and \cite{Sch1} for an alternative treatment using the Buchsbaum-Eisenbud structure theorem of Gorenstein rings of codimension $3$).  Geiss \cite{G} using liaison techniques showed that most Hurwitz spaces $\H_{g,6}$ with $g\leq 45$ are unirational. Schreyer and Tanturri \cite{ST} put forward the hypothesis that there exist only finitely many pairs $(g,k)$, with $k\geq 6$, such that $\hh_{g,k}$ is \emph{not} unirational. At the other end, when $k=g$, the Hurwitz space $\hh_{g,g}$ is birationally isomorphic to the universal symmetric product of degree $g-2$ over $\mm_g$ and it has been showed in \cite{FV2} that $\hh_{g,g}$ is of general type for $g\geq 12$ and uniruled in all the other cases. Further results on the unirationality of some particular Hurwitz spaces are obtained in \cite{KT} and \cite{DS}.

\vskip 5pt

In this paper, we are particularly interested in the range $\frac{g+2}{2}\leq k\leq g-1$ (when $\hh_{g,k}$ dominates $\mm_g$) and $14\leq g\leq 19$, so that $\mm_g$ is either unirational/uniruled in the cases $g=14, 15$, or its Kodaira dimension is unknown, when $g=16, 17, 18, 19$. We summarize our results showing the positivity of the canonical bundle of $\hh_{g,k}$ and we begin with the cases when $\mm_g$ is known to be (or thought to be) uniruled.

\begin{theorem}\label{kodhurw1}
Suppose $\hh$ is one of the spaces  $\hh_{14,9}$ or $\hh_{16,9}$. Then  there exists an effective $\mathbb Q$-divisor class $E$ on $\hh$ supported on the union $\sum_{i^\geq 3} \mathfrak{b}^*(\widetilde{B}_i)$, such that the twisted canonical class $K_{\hh}+E$ is effective.
\end{theorem}

A few comments are in order. We expect that the boundary divisor $E$ is empty. Carrying this out requires an extension of the determinantal structure defining the Koszul divisors considered in this paper to the entire space $\hh_{g,k}$. This technically cumbersome step will be carried out in future work. Secondly, Theorem \ref{kodhurw1} is optimal. In genus $14$, Verra \cite{Ve} established the unirationality of  the Hurwitz space $\H_{14,8}$, which is a finite cover of $\cM_{14}$, and concluded in this way that $\cM_{14}$ itself is unirational. In genus $16$, the map $$\sigma\colon \hh_{16,9}\rightarrow \mm_{16}$$ is generically finite. Chang and Ran \cite{CR3} claimed that $\mm_{16}$ is uniruled (more precisely, that $K_{\mm_{16}}$ is not pseudo-effective, which by now, one knows that it implies uniruledness). However, it has become recently clear \cite{Ts} that their proof contains a fatal error, so the Kodaira dimension of $\mm_{16}$ is currently unknown.

\vskip 5pt

\begin{theorem}\label{kodhurw2}
Suppose $\hh$ is one of the spaces $\hh_{17,11}$ or $\hh_{19,13}$. Then there exists an effective $\mathbb Q$-divisor class $E$ on $\hh$ supported on the union $\sum_{i\geq 3} \mathfrak{b}^*(\widetilde{B}_i)$, such that  the twisted canonical class $K_{\hh}+E$ is big.
\end{theorem}

 We observe that the maps
$\hh_{17,11}\rightarrow \mm_{17}$ and $\hh_{19,13}\rightarrow \mm_{19}$ have generically $3$ and respectively $5$-dimensional fibres. As in the case of Theorem \ref{kodhurw1}, here too we expect the divisor $E$ to be empty.

\vskip 3pt

Both Theorems \ref{kodhurw1} and \ref{kodhurw2} are proven at the level of a partial compactification $\widetilde{\G}_{g,k}^1$ (described in detail in Section 2) and which incorporates only admissible covers with a source curve whose stable model is irreducible. In each relevant genus we produce an explicit effective divisor $\widetilde{\cD}_g$ on $\widetilde{\G}_{g,k}^1$  such that the canonical class $K_{\widetilde{\G}_{g,k}^1}$ can be expressed as a positive combination of a multiple of $[\widetilde{\cD}_g]$ and the pull-back under the map $\sigma$ of an effective divisor on $\mm_g$. Theorems \ref{kodhurw1} and \ref{kodhurw2} offer a very strong indication that the moduli spaces
$\hh_{14,9}, \hh_{16,9}, \hh_{17,11}$ and $\hh_{19,13}$ are all of general type. To complete such a proof one would need to show that the singularities of $\hh_{g,k}$ impose no \emph{adjunction conditions}, that is, every pluricanonical form on $\hh_{g,k}$ lifts to a pluricanonical form of a \emph{resolution of singularities} of $\hh_{g,k}$. Establishing such a result, which holds at the level of $\mm_g$, see \cite{HM}, or at the level of the moduli space of $\ell$-torsion points $\rr_{g,\ell}$ for small levels \cite{CF}, remains a significant challenge.

\vskip 5pt

We describe the construction of these divisors in two instances and refer to Section 2 for the remaining cases and further details. First we consider the space $\hh_{16,9}$ which is a generically finite cover of $\mm_{16}$.

\vskip 4pt

Let us choose a general
pair $[C,A]\in \H_{16,9}$. We set $L:=K_C\otimes A^{\vee}\in W^7_{21}(C)$. Denoting by $I_{C,L}(k):=\mbox{Ker}\bigl\{\mbox{Sym}^k H^0(C,L)\rightarrow H^0(C,L^{\otimes k})\bigr\}$, by Riemann-Roch one computes that $\mbox{dim } I_{C,L}(2)=9$ and $\mbox{dim } I_{C,L}(3)=72$. Therefore the locus where there exists a non-trivial syzygy, that is, the map
$$\mu_{C,L}\colon I_{C,L}(2)\otimes H^0(C,L)\rightarrow I_{C,L}(3)$$ is not an isomorphism is expected to be a divisor on $\H_{16,9}$. This indeed is the case and we define the \emph{syzygy divisor}
$$\cD_{16}:=\Bigl\{[C,A]\in \H_{16,9}: I_{C,K_C\otimes A^{\vee}}(2)\otimes H^0(C,K_C\otimes A^{\vee})\stackrel{\ncong}\rightarrow I_{C,K_C\otimes A^{\vee}}(3)\Bigr\}.$$ The ramification divisor of the map $\sigma\colon \H_{16,9}\rightarrow \cM_{16}$, viewed as the
\emph{Gieseker-Petri divisor}
$$\mathcal{GP}:=\Bigl\{[C,A]\in \H_{16,9}\colon  H^0(C,A)\otimes H^0(C,K_C\otimes A^{\vee})\stackrel{\ncong}\rightarrow H^0(C,K_C)\Bigr\}$$
also plays an important role. After computing the class of the closure $\widetilde{\cD}_{16}$ of the Koszul divisor $\cD_{16}$ inside the partial
compactification $\widetilde{\G}^1_{16,9}$, we find the following explicit representative for the canonical class
\begin{equation}\label{canrepre}
K_{\widetilde{\G}^1_{16,9}}=\frac{2}{5}[\widetilde{\cD}_{16}]+\frac{3}{5}[\widetilde{\mathcal{GP}}]\in CH^1(\widetilde{\G}^1_{16,9}).
\end{equation}
This proves Theorem \ref{kodhurw1} in the case $g=16$. We may wonder whether (\ref{canrepre}) is the only effective representative of $K_{\widetilde{\G}^1_{16,9}}$, which would imply that the Kodaira dimension of $\hh_{16,9}$ is equal to zero. To summarize, since the smallest Hurwitz cover of $\mm_{16}$ has an effective canonical class, it seems unlikely that the method of establishing the uniruledness/unirationality of $\cM_{14}$ and $\cM_{15}$ by studying a Hurwitz space covering it can be extended to higher genera $g\geq 17$.

\vskip 5pt

The last  case we discuss in this introduction is $g=17$ and $k=11$. We choose a general pair $[C,A]\in \H_{17,11}$. The residual linear system $L:=K_C\otimes A^{\vee}\in W^6_{21}(C)$ induces
an embedding $C\subseteq \PP^6$ of degree $21$. The multiplication map
$$\phi_L\colon \mbox{Sym}^2 H^0(C,L)\rightarrow H^0(C,L^{\otimes 2})$$
has a $2$-dimensional kernel. We impose the condition that the pencil of quadrics containing the image curve $C\stackrel{|L|}\hookrightarrow \PP^6$ be \emph{degenerate},
that is, it intersects the discriminant divisor in $\PP\bigl(\mbox{Sym}^2 H^0(C,L)\bigr)$ non-transversally. We thus define the locus
$$\cD_{17}:=\Bigl\{[C,A]\in \H_{17,11}: \PP\bigl(\mbox{Ker}(\phi_{K_C\otimes A^{\vee}})\bigr) \ \mbox{ is a degerate pencil}\Bigr\}.$$
Using \cite{FR}, we can compute the class $[\widetilde{\cD}_{17}]$ of the closure of $\cD_{17}$ inside $\widetilde{\G}_{17,11}$. Comparing this class to that of the canonical divisor, we obtain the relation
\begin{equation}\label{gen17}
K_{\widetilde{\G}_{17,11}^1}=\frac{1}{5}[\widetilde{\cD}_{17}]+\frac{3}{5}\sigma^*(7\lambda-\delta_0)\in CH^1(\widetilde{\G}^1_{17,11}).
\end{equation}
Since the class $7\lambda-\delta_0$ can be easily shown  to be big on $\mm_{17}$, the conclusion of Theorem \ref{kodhurw2} in the case $g=17$ now follows.

\section{The birational geometry of Hurwitz spaces}

We denote by  $\mathcal{H}_{g,k}^{\mathfrak o}$ the Hurwitz space classifying degree $k$ covers $f\colon C\rightarrow \PP^1$ with source being a smooth curve $C$ of genus $g$ and having simple ramifications. Note that we choose  an \emph{ordering} $(p_1, \ldots, p_{2g+2k-2})$ of the set of the branch points of $f$. An excellent reference for the algebro-geometric study of Hurwitz spaces is \cite{Ful}.
Let $\hh_{g,k}^{\mathfrak o}$ be the (projective) moduli space of admissible covers (with an ordering of the set of branch points). The geometry of $\hh_{g,k}^{\mathfrak o}$ has been described in detail by Harris and Mumford \cite{HM} and further clarified in \cite{ACV}. Summarizing their results, the stack $\overline{H}_{g,k}^{\mathfrak o}$ of admissible covers (whose coarse moduli space is precisely $\hh_{g,k}^{\mathfrak o}$) is isomorphic to the stack of \emph{twisted stable} maps into the classifying stack $\mathcal{B} \mathfrak{S}_{k}$ of the symmetric group $\mathfrak S_{k}$, that is, there is a canonical identification
$$\overline{H}_{g,k}^{\mathfrak o}:=\overline{M}_{0,2g+2k-2}\Bigl(\mathcal{B} \mathfrak S_{k}\Bigr).$$
Points of  $\hh_{g,k}^{\mathfrak o}$ can be thought of as admissible covers
$[f\colon C\rightarrow R, p_1, \ldots, p_{2g+2k-2}]$, where the source $C$ is a nodal curve of arithmetic genus $g$, the target $R$ is a tree of smooth rational curves, $f$ is a finite map of degree $k$ satisfying $f^{-1}(R_{\mathrm{sing}})=C_{\mathrm{sing}}$, and $p_1, \ldots, p_{2g+2k-2}\in R_{\mathrm{reg}}$
denote the branch points of $f$. Furthermore, the ramification indices on the two branches of $C$ at each ramification point of $f$ at a node of $C$ must coincide. One has a finite \emph{branch} morphism
$$\mathfrak{b}\colon \hh_{g,k}^{\mathfrak o}\rightarrow \mm_{0, 2g+2k-2},$$ associating to a cover its (ordered) branch locus. The symmetric group
$\mathfrak S_{2g+2k-2}$ operates on $\hh_{g,k}^{\mathfrak{o}}$ by permuting the branch points of each admissible cover. Denoting by
$$\hh_{g,k}:=\hh_{g,k}^{\mathfrak o}/\mathfrak S_{2g+2k-2}$$ the quotient parametrizing admissible covers \emph{without} an ordering of the branch points, we introduce the projection $q\colon \hh_{g,k}^{\mathfrak o} \rightarrow \hh_{g,k}$. Finally, let
$$\sigma\colon \hh_{g,k}\rightarrow \mm_{g}$$
be the map assigning to an admissible degree $k$ cover the stable model of its source curve, obtained by contracting unstable rational components.

\vskip 5pt

\subsection{Boundary divisors of $\hh_{g,k}$.}
We discuss the structure of the boundary divisors on the compactified Hurwitz space. For $i=0, \ldots, g+k-1$, let $B_i$ be the boundary divisor of $\mm_{0,2g+2k-2}$, defined as the closure of the locus of unions of two smooth rational curves meeting at one point, such that precisely $i$ of the marked points lie on one component and $2g+2k-2-i$ on the remaining one. To specify a  boundary divisor of $\hh_{g,k}^{\mathfrak o}$, one needs  the following combinatorial information:
\begin{enumerate}
\item A partition $I\sqcup J=\{1,\dotsc,2g+2k-2\}$, such that $|I|\geq 2$ and
$|J|\geq 2$.
\item Transpositions $\{w_i\}_{i\in I}$ and $\{w_j\}_{j\in J}$ in $\mathfrak S_k$, satisfying $$\prod_{i\in I} w_i = u, \ \ \prod_{j\in J}w_j=u^{-1},$$ for some
permutation $u\in \mathfrak S_k$.
\end{enumerate}

To this data (which is determined up to conjugation with an element from $\mathfrak{S}_k$), we associate the locus of admissible covers of degree $k$ with labeled branch points
$$\bigl[f\colon C\rightarrow R, \ p_1, \ldots, p_{2g+2k-2}\bigr]\in \hh_{g,k}^{\mathfrak o},$$
where $[R=R_1\cup_p R_2, p_1, \ldots, p_{2g+2k-2}]\in B_{|I|}\subseteq \mm_{0,2g+2k-2}$ is a pointed union of two smooth rational curves $R_1$ and $R_2$ meeting at the point $p$. The marked points indexed by $I$ lie on $R_1$, those indexed by $J$ lie on $R_2$. Let $\mu:=(\mu_1, \ldots, \mu_{\ell})\vdash k$ be the partition corresponding to the conjugacy class of $u\in \mathfrak S_k$. We denote by $E_{i:\mu}$ the boundary divisor on $\hh_k^{\mathfrak o} $ classifying twisted stable maps with underlying admissible cover as above, with $f^{-1}(p)$ having partition type $\mu$, and exactly $i$ of the points $p_1, \ldots, p_{2g+2k-2}$ lying on the component $R_1$. Passing to the unordered Hurwitz space $\hh_{g,k}$, we denote by $D_{i:\mu}$ the image $E_{i:\mu}$ under the map $q$, with its reduced structure. In general, both $D_{i:\mu}$ and $E_{i:\mu}$ may well be reducible.

The effect of the  map $\mathfrak{b}$ on boundary divisors is summarized in the following relation that holds for $i=2, \ldots, g+k-1$, see \cite{HM} page 62, or \cite{GK1} Lemma 3.1:
\begin{equation}\label{pb}
\mathfrak{b}^*(B_i)=\sum_{\mu\vdash k} \mbox{lcm}(\mu)E_{i:\mu}.
\end{equation}

\vskip 5pt

The Hodge class on the Hurwitz space is by definition pulled back from $\mm_g$. Its class $\lambda:=(\sigma\circ q)^*(\lambda)$ on $\hh_{g,k}^{\mathfrak o}$ has been determined in \cite{KKZ}, or see \cite{GK1} Theorem 1.1 for an algebraic proof. Remarkably, unlike on $\mm_g$, the Hodge class is always a boundary class:
\begin{equation}\label{hodgecl}
\lambda=\sum_{i=2}^{g+k-1} \sum_{\mu\vdash k} \mathrm{lcm}(\mu)\Bigl(\frac{i(2g+2k-2-i)}{8(2g+2k-3)}-\frac{1}{12}\Bigl(k-\sum_{j=1}^{\ell(\mu)} \frac{1}{\mu_j}\Bigr)\Bigr)[E_{i:\mu}]\in CH^1(\hh_k^{\mathfrak{o}}).
\end{equation}

The sum (\ref{hodgecl}) runs  over  partitions $\mu$ of $k$ corresponding to conjugacy classes of  permutations that can be written as products of $i$ transpositions. In the formula (\ref{hodgecl}) $\ell(\mu)$ denotes the length of the partition $\mu$.  The only negative coefficient in the expression of $K_{\mm_{0,2g+2k-2}}$ is that of the boundary divisor $B_2$. For this reason, the components of $\mathfrak{b}^*(B_2)$ play a special role, which we now discuss. We pick an admissible cover $$[f\colon C=C_1\cup C_2\rightarrow R=R_1\cup_p R_2,\  p_1, \ldots, p_{2g+2k-2}]\in \mathfrak{b}^*(B_2),$$ and set $C_1:=f^{-1}(R_1)$ and $C_2:=f^{-1}(R_2)$ respectively. Note that $C_1$ and $C_2$ may well be disconnected. Without loss of generality, we assume $I=\{1, \ldots, 2g+2k-4\}$, thus $p_1, \ldots, p_{2g+2k-4}\in R_1$ and $p_{2g+2k-3}, p_{2g+2k-2}\in R_2$.

\vskip 4pt

Let $E_{2:(1^k)}$ be the closure in $\hh_{g,k}^{\mathfrak o}$ of the locus of admissible covers such that the transpositions $w_{2g+2k-3}$ and $w_{2g+2k-2}$ describing the local monodromy in a neighborhood of the branch points $p_{2g+2k-3}$ and $p_{2g+2k-2}$ respectively, are equal. Let $E_0$ further denote the subdivisor of $E_{2:(1^k)}$ consisting of those admissible cover for which the subcurve $C_1$ is connected. This is the case precisely when
$\langle w_1, \ldots, w_{2g+2k-4}\rangle =\mathfrak S_k$.  Note that $E_{2:(1^k)}$ has many components not contained in $E_0$, for instance when $C_1$ splits as the disjoint union of a smooth rational curve mapping isomorphically onto $R_1$ and a second component mapping with degree $k-1$ onto $R_1$.

\vskip 3pt

When the permutations $w_{2g+2k-3}$ and $w_{2g+2k-2}$ are distinct but share one element in their orbit, then $\mu=(3,1^{k-3})\vdash k$ and the corresponding boundary divisor is denoted by $E_{2: (3,1^{k-3})}$. Let $E_3$ be the subdivisor of $E_{2:(3,1^{k-3})}$ corresponding to admissible covers with  $\langle w_1, \ldots, w_{2g+2k-4}\rangle =\mathfrak S_k$, that is, $C_1$ is a connected curve. Finally, in the case when $w_{2g+2k-3}$ and $w_{2g+2k-2}$ are disjoint transpositions, we obtain the boundary divisor $E_{2:(2,2,1^{k-4})}$. Similarly to the previous case, we denote by $E_2$ the subdivisor of $E_{2:(2,2,1^{k-4})}$ consisting of admissible covers  for which  $\langle w_1, \ldots, w_{2g+2k-4}\rangle =\mathfrak S_k$.

\vskip 4pt

We denote by $D_0, D_2$ and $D_3$ the push-forward of $E_0, E_2$ and $E_3$ respectively under the map $q$. Using a variation of Clebsch's original argument \cite{Cl} for establishing the irreducibility of $\hh_{g,k}$, it is easy to establish that $D_0$, $D_2$ and $D_3$ are all irreducibile, see also \cite{GK3} Theorem 6.1. The boundary divisors $E_0, E_2$ and $E_3$, when pulled-back under the quotient map $q\colon \hh_{g,k}^{\mathfrak{o}}\rightarrow \hh_{g,k}$, verify the following formulas
$$q^*(D_0)=2E_0, \ q^*(D_2)=E_2 \ \mbox{ and } q^*(D_3)=2E_3,$$
which we now explain. The general point of both $E_0$ and $E_3$ has no automorphism that fixes all branch points, but admits an automorphism of order two that fixes $C_1$ and permutes the branch points $p_{2g+2k-3}$ and $p_{2g+2k-2}$. The general admissible cover  in $E_2$ has an automorphism group $\mathbb Z_2\times \mathbb Z_2$ (each of the two components of $C_2$ mapping $2:1$ onto $R_2$ has an automorphism of order $2$). In the stack $\overline{H}_{g,k}^{\mathfrak o}$ we have two points lying over this admissible cover and each of them has an automorphism group of order $2$. In particular the map $\overline{H}_{g,k}^{\mathfrak o}\rightarrow \hh_{g,k}^{\mathfrak o}$ from the stack to the coarse moduli space is ramified with ramification index $1$ along the divisor $E_2$.

\vskip 4pt

One applies now the Riemann-Hurwitz formula to the map $\mathfrak{b}\colon \hh_{g,k}^{\mathfrak o}\rightarrow \mm_{0,2g+2k-2}$. Recall also that the canonical bundle of the moduli space of pointed rational curves is given by the formula
$$K_{\mm_{0,2g+2k-2}}= \sum_{i=2}^{g+k-1} \Bigl(\frac{i(2g+2k-2-i)}{2g+2k-3}-2\Bigr)[B_i].$$ All in all, we obtain the following formula for the canonical class of the Hurwitz stack:
\begin{equation}\label{canhur}
K_{\overline{H}_{g,k}^{\mathfrak o}}=\mathfrak{b}^*K_{\mm_{0, 2g+2k-2}}+\mbox{Ram}(\mathfrak{b}),
\end{equation}
where $\mbox{Ram}(\mathfrak{b})=\sum_{i, \mu\vdash k} (\mathrm{lcm}(\mu)-1)[E_{i:\mu}]$.

\vskip 4pt

\subsection{A partial compactification of $\H_{g,k}$.}
Like in the paper \cite{FR}, we work on a partial compactification of $\H_{g,k}$, for the Koszul-theoretic calculations are difficult to extend over all the boundary divisors $D_{i:\mu}$. The partial compactification of the Hurwitz space we consider is defined in  the same spirit as in the part of the paper devoted to divisors on $\mm_{g}$. We fix an integer
$k\leq \frac{2g+4}{3}$. Then $\rho(g,2,k)<-2$ and using \cite{EH2} it follows then that the locus of curves $[C]\in \cM_g$ such that $W^2_k(C)\neq \emptyset$ has codimension at least $2$ in $\cM_g$.  We denote by $\widetilde{\G}_{g,k}^1$ the space of pairs $[C,A]$, where $C$ is an irreducible nodal curve of genus $g$ satisfying $W^2_k(C)=\emptyset$ and $A$ is a base point free locally free sheaf of rank $1$ and degree $k$ on $C$ with $h^0(C,A)=2$. The rational map $$\hh_{g,k}\dashrightarrow \widetilde{\G}_{g,k}^1$$ is of course regular outside a subvariety of $\hh_{g,k}$ of codimension at least $2$, but  can be made explicit over the boundary divisors $D_0, D_2$ and $D_3$, which we now explain.

\vskip 3pt

Retaining the previous notation, to the general point $[f\colon C_1\cup C_2\rightarrow R_1\cup_p R_2]$ of $D_3$ (respectively $D_2$), we assign the pair $[C_1, A_1:=f^*\OO_{R_1}(1)]\in \widetilde{\G}_k^1$. Note that $C_1$ is a smooth curve of genus $g$ and $A_1$ is a pencil on $C_1$ having a triple point (respectively two ramification points in the fibre over $p$).  The spaces $\H_{g,k}\cup D_0\cup D_2\cup D_3$ and $\widetilde{\G}_{g,k}^1$ differ outside a set of codimension at least $2$ and for divisor class calculations they will be identified. Using this, we copy the formula (\ref{hodgecl}) at the level of the parameter space $\widetilde{\G}_{g,k}^1$ and obtain:
\begin{equation}\label{hodgepart}
\lambda=\frac{g+k-2}{4(2g+2k-3)}[D_0]-\frac{1}{4(2g+2k-3)}[D_2]+\frac{g+k-6}{12(2g+2k-3)}[D_3]\in CH^1(\widetilde{\G}_{g,k}^1).
\end{equation}

We now observe that the canonical class of $\widetilde{\G}_k^1$ has a simple expression in terms of the Hodge class $\lambda$ and the boundary divisors $D_0$ and $D_3$. Quite remarkably, this formula is independent of both $g$ and $k$!

\begin{theorem}\label{canpart2}
The canonical class of the partial compactification $\widetilde{\G}_{g,k}^1$ is given by
$$K_{\widetilde{\G}_{g,k}^1}=8\lambda+\frac{1}{6}[D_3]-\frac{3}{2}[D_0].$$
\end{theorem}
\begin{proof} We combine the equation (\ref{canhur}) with the Riemann-Hurwitz formula applied to the quotient $q\colon \hh_{g,k}^{\mathfrak o}\dashrightarrow \widetilde{\G}_{g,k}^1$ and write:
$$q^*(K_{\widetilde{\G}_{g,k}^1})=K_{\overline{H}_{g,k}^{\mathfrak o}}-[E_0]-[E_2]-[E_3]=
$$
$$-\frac{4}{2g+2k-3}[D_2]-\frac{2g+2k-1}{2g+2k-3}[D_0]+\frac{2g+2k-9}{2g+2k-3}[D_3].$$

To justify this formula, observe that the  divisors $E_0$ and $E_3$ lie in the ramification locus of $q$, hence they have to subtracted from $K_{\overline{H}_{g,k}^{\mathfrak 0}}$. The morphism $\overline{H}_{g,k}^{\mathfrak{o}}\rightarrow \hh_{g,k}^{\mathfrak o}$ from the stack to the coarse moduli space is furthermore simply ramified along $E_2$, so this divisor has to be subtracted as well.
We now use (\ref{hodgepart}) to express  $[D_2]$ in terms of $\lambda$, $[D_0]$ and $[D_3]$ and obtain that $q^*(K_{\widetilde{\G}^1_{g,k}})=8\lambda+\frac{1}{3} [E_3]-3[E_0]$, which yields the claimed formula.
\end{proof}

\vskip 4pt

Let $f\colon \cC\rightarrow \widetilde{\G}_{g,k}^1$ be the universal curve and we choose a degree $k$ Poincar\'e line bundle $\mathcal{L}$ on $\cC$ (or on an \'etale cover if necessary). Along the lines of \cite{FR} Section 2 (where only the case $g=2k-1$ has been treated, though the general situation is analogous), we introduce two  tautological codimension one classes:
$$\mathfrak{a}:=f_*\bigl(c_1^2(\mathcal{L})\bigr) \mbox{ } \mbox{and  } \mbox{ } \mathfrak{b}:=f_*\bigl(c_1(\mathcal{L})\cdot c_1(\omega_f)\bigr)\in  CH^1(\widetilde{\G}_{g,k}^1).$$
The push-forward sheaf $\mathcal{V}:=f_*\mathcal{L}$ is locally free of rank $2$ on $\widetilde{\G}_{g,k}^1$ (outside a subvariety of codimension at least $2$). Its fibre at a point $[C,A]$ is canonically identified with $H^0(C,A)$. Although  $\mathcal{L}$ is not unique, an easy exercise involving first Chern classes, convinces us that the class
\begin{equation}\label{gamma}
\gamma:=\mathfrak b-\frac{g-1}{k}\mathfrak a \in CH^1(\widetilde{\G}_{g,k}^1)
\end{equation}
does not depend of the choice of a Poincar\'e bundle.
\begin{proposition}\label{basepoint}
We have that $\mathfrak{a}=kc_1(\cV)\in CH^1(\widetilde{\G}_{g,k}^1)$.
\end{proposition}
\begin{proof} Simple application of the Porteous formula in the spirit of Proposition 11.2 in \cite{FR}.
\end{proof}

The following locally free sheaves on $\widetilde{\G}_{g,k}^1$ will play an important role in several Koszul-theoretic calculations:
$$\E:=f_*(\omega_f\otimes \L^{-1}) \ \mbox{ and } \ \F_{\ell}:=f_*(\omega_f^{\ell}\otimes \L^{-\ell}),$$
where $\ell\geq 2$.
\begin{proposition}\label{chernhurw1}
The following formulas hold
$$c_1(\E)=\lambda-\frac{1}{2}\mathfrak b+\frac{k-2}{2k} \mathfrak a \ \mbox{ and } \ c_1(\F_{\ell})=\lambda+\frac{\ell^2}{2}\mathfrak{a}-\frac{\ell(2\ell-1)}{2}\mathfrak{b}+{\ell \choose 2}\bigl(12\lambda-[D_0]\bigr).$$
\end{proposition}
\begin{proof} Use Grothendieck-Riemann-Roch applied to the universal curve $f$, coupled with Proposition \ref{basepoint} in order to evaluate the terms. Use that $R^1f^*(\omega_f^{\ell}\otimes \L^{-\ell})=0$ for $\ell\geq 2$. Similar to Proposition 11.3 in \cite{FR}, so we skip the details.
\end{proof}

We summarize the relation between the class $\gamma$ and the classes $[D_0], [D_2]$ and $[D_3]$ as follows. Again, we find it remarkable that this formula is independent of $g$ and $k$.

\begin{proposition}\label{trans}
One has the formula $[D_3]=6\gamma+24\lambda-3[D_0].$
\end{proposition}
\begin{proof}
 We form the fibre product of the universal curve $f\colon \mathcal{C}\rightarrow \widetilde{\G}_{g,k}^1$ together with its projections:
$$
\begin{CD}
{\mathcal{C}} @<\pi_1<< {\mathcal C\times_{\widetilde{\G}_{g,k}^1}\mathcal{C}} @>\pi_2>> {\mathcal{C}} \\
\end{CD}.
$$
For $\ell\geq 1$, we consider the jet bundle $J_f^{\ell}(\L)$, which sits in an  exact sequence:
\begin{equation}\label{exseqjet}
0 \longrightarrow  \omega_f^{\otimes \ell}\otimes \mathcal{L} \longrightarrow J_f^{\ell}(\mathcal{L}) \longrightarrow       J_f^{\ell-1}(\mathcal{L}) \longrightarrow 0.
\end{equation}
Here $J_f^{\ell}(\L):=\bigl(\mathcal{P}_f^{\ell}(\L)\bigr)^{\vee \vee}$ is the \emph{reflexive closure} of the sheaf of principal parts
defined as $\mathcal{P}_f^{\ell}(\L):=(\pi_2)_* \bigl(\pi_1^*(\L)\otimes \mathcal{I}_{(\ell+1)\Delta}\bigr)$, where $\Delta \subseteq \mathcal C\times_{\widetilde{\G}_{g,k}^1}\mathcal{C}$ is the diagonal.

\vskip 3pt

One has a sheaf morphism $\nu_{2}\colon f^*(\cV)\rightarrow J_f^{2}(\L)$, which we think of as the \emph{second Taylor map} associating to a section its first two derivatives. For points $[C,A,p]\in \mathcal{C}$ such that $p\in C$ is a smooth point,  this map is simply the evaluation $H^0(C,A)\rightarrow H^0(A\otimes \OO_{3p})$. Let $Z\subseteq \mathcal{C}$ be the locus where $\nu_2$ is not injective. Over the locus of smooth curves, $D_{3}$ is the set-theoretic image of $Z$. A local analysis that we shall present shows that $\nu_2$ is degenerate with multiplicity $1$ at a point $[C,A,p]$, where $p\in C_{\mathrm{sing}}$. Thus, $D_0$ is to be found with multiplicity $1$ in the degeneracy locus of $\nu_2$. The  Porteous formula leads to:
$$
[D_3]= f_* c_2\Bigl(\frac{J_f^2(\mathcal{L})}{f^*(\cV)}\Bigr)-[D_{0}]\in CH^1(\widetilde{\G}^1_{g,k}).
$$

As anticipated, we now show that $D_0$ appears with multiplicity $1$ in the degeneracy locus of $\nu_2$. To that end, we choose a family $F:X \rightarrow B$ of genus $g$ curves of genus over a smooth $1$-dimensional base $B$, such that $X$ is smooth, and there is a point $b_0\in B$ with $X_b:=F^{-1}(b)$ is smooth for $b\in B\setminus \{b_0\}$, whereas $X_{b_0}$ has a unique node $u\in X$. Assume furthermore  that $A\in \mbox{Pic}(X)$ is a line bundle such that $A_b:=L_{|X_b}\in W^1_k(X_b)$, for each $b\in B$.  We further choose a local parameter $t\in \mathcal{O}_{B, b_0}$ and $x, y\in \mathcal{O}_{X,u}$, such that $xy=t$ represents the local equation of $X$ around the point $u$. Then $\omega_F$ is locally generated by the meromorphic differential $\tau$ that is given by $\frac{dx}{x}$ outside the divisor $x=0$ and by $-\frac{dy}{y}$ outside the divisor $y=0$. Let us pick sections $s_1, s_2\in H^0(X,A)$, where $s_1(u)\neq 0$, whereas $s_2$ vanishes with order $1$ at the node $u$ of $X_{b_0}$, along  both its branches. Passing to germs of functions at $u$, we have the relation $s_{2,u}=(x+y)s_{1,u}$. Then by direct calculation in local
coordinates, the map $H^0\bigl(X_{b_0}, A_{b_0}\bigr) \rightarrow H^0\bigl(X_{b_0}, A_{b_0|3u}\bigr)$ is given by the $2\times 2$ minors of the following matrix:
$$\begin{pmatrix}
1 & 0& 0\\
x+y & x-y & x+y \\
\end{pmatrix}.
$$
We conclude that $D_0$ appears with multiplicity $1$ in the degeneracy locus of $\nu_2$.

\vskip 4pt

From the exact sequence (\ref{exseqjet}) one computes $c_1(J_f^2(\L))=3c_1(\L)+3c_1(\omega_f)$ and $c_2(J_f^2(\L))=c_2(J_f^1(\L))+c_1(J_f^1(\L))\cdot c_1(\omega_f^{\otimes 2}\otimes \L)=3c_1^2(\L)+6c_1(\L)\cdot c_1(\omega_f)+2c_1^2(\omega_f)$.
Substituting, we find after routine calculations that
$$f_* c_2\left(\frac{J_f^2(\mathcal{L})}{f^*(\cV)}\right)=6\gamma+2\kappa_1,$$
where  $\kappa_1=f_*(c_1(\omega_f)^{\otimes 2})$.
Using Mumford's formula $\kappa_1=12\lambda-[D_0]\in CH^1(\widetilde{\G}_{g,k}^1)$, see e.g. \cite{HM} top of page 50, we finish the proof.
\end{proof}

\section{Effective divisors on Hurwitz spaces when $14\leq g\leq 19$}

We now describe the construction of four effective divisors on particular Hurwitz spaces interesting in moduli theory. The divisors in question are of syzygetic nature and
resemble somehow the divisors on $\mm_g$ used recently in \cite{FJP} to show that both $\mm_{22}$ and $\mm_{23}$, as well as those used earlier in \cite{F2} and \cite{F3} to disprove the Slope Conjecture formulated in \cite{HMo}.  The divisors on $\hh_{g,k}$ that we consider
are defined directly in terms of a general element $[C,A]\in \H_{g,k}$ (where $A\in W^1_k(C)$), without making reference to other attributes of the curve $C$. This simplifies both the task of computing their classes and showing that the respective codimension $1$ conditions in moduli lead to genuine divisor on $\H_{g,k}$. Using the irreducibility of $\H_{g,k}$, this amounts to exhibiting \emph{one} example of a point $[C,A]\in \H_{g,k}$ outside the divisor, which can be easily achieved with the use of
\emph{Macaulay}. This is in sharp contrast to the situation in \cite{FJP}, where establishing the transversality statement ensuring that the virtual divisor on the moduli spaces in question are actual divisors turns out to be a major challenge. We mention finally also the papers  \cite{GK2}, \cite{DP}, where one studies other interesting divisors  in Hurwitz spaces using the splitting type of the $(k-1)$-dimensional scroll canonically associated to a degree $k$ cover $C\rightarrow \PP^1$.

\vskip 3pt

\subsection{The Hurwitz space $\H_{14,9}$.} We consider the morphism $\sigma\colon \hh_{14,9}\rightarrow \mm_{14}$, whose general fibre is $2$-dimensional. We choose a general element  $[C,A]\in \H_{14,9}$ and set
$L:=K_C\otimes A^{\vee}\in W^5_{17}(C)$ to be the residual linear system. Furthermore, $L$ is very ample, else there exist points $x,y\in C$ such that $A(x+y) \in W^2_{13}(C)$, which contradicts the Brill-Noether Theorem. Note that
$$h^0(C,L^{\otimes 2})=\mbox{dim } \mbox{Sym}^2 H^0(C,L)=21$$ and we set up the (a priori virtual) divisor
$$\widetilde{\cD}_{14}:=\Bigl\{[C,A]\in \widetilde{\G}^1_{14,9}:  \mbox{Sym}^2 H^0(C,L)\stackrel{\phi_L}\longrightarrow H^0(C,L^{\otimes 2}) \mbox{ is not an isomorphism}\Bigr\}.$$

Our next results shows that, remarkably, this locus is indeed a divisor  and it gives rise to an  effective representative of the canonical divisor of $\widetilde{\G}_{14,9}^1$.
\begin{proposition}\label{gen14class}
The locus $\cD_{14}$ is a divisor on $\H_{14,9}$ and one has the following formula
$$[\widetilde{\cD}_{14}]=4\lambda+\frac{1}{12}[D_3]-\frac{3}{4}[D_0]=\frac{1}{2} K_{\widetilde{\G}^1_{14,9}}\in CH^1(\widetilde{\G}^1_{14,9}).$$
\end{proposition}
\begin{proof} The locus $\widetilde{\cD}_{14}$ is the degeneracy locus of the vector bundle morphism
\begin{equation}\label{morphism:1}
\phi\colon \mbox{Sym}^2(\E)\rightarrow \F_2.
\end{equation}
The Chern class of both vector bundles $\E$ and $\F_2$ are computed in Proposition \ref{chernhurw1} and we have the formulas $c_1(\E)=\lambda-
\frac{1}{2}\mathfrak{b}+\frac{7}{18}\mathfrak{a}$ and $c_1(\F_2)=13\lambda+2\mathfrak{a}-3\mathfrak{b}-[D_0]$. Taking into account that $\mbox{rk}(\E)=6$,  we can write:
$$[\widetilde{\cD}_{14}]=c_1\bigl(\F_2-\mbox{Sym}^2(\E)\bigr)=c_1(\F_2)-7c_1(\E)=6\lambda-[D_0]+\frac{1}{2}\gamma.$$
We now substitute $\gamma$ in the formula given by Proposition \ref{trans} involving also the divisor $D_3$ on $\widetilde{\G}^1_{14,9}$
of pairs $[C,A]$, such that $A\in W^1_9(C)$ has a triple ramification point. We obtain that $[\widetilde{\cD}_{14}]=4\lambda+\frac{1}{12}[D_3]-\frac{3}{4}[D_0]$.
Comparing with Theorem \ref{canpart2}, the fact that $[\widetilde{\cD}_{14}]$ is a (half-) canonical representative on $\widetilde{\mathcal{G}}^1_{14,9}$ now follows.

\vskip 4pt

It remains to show that $\widetilde{\cD}_{14}$ is indeed a divisor, that is, the morphism $\phi$ given by (\ref{morphism:1}) is generically non-degenerate. Since $\H_{14,9}$ is irreducible, it suffices to construct one example of a smooth curve $C\subseteq \PP^5$ of genus $14$ and degree $17$ which does not lie on any quadrics. To that end, we consider the \emph{White surface}  $X\subseteq \PP^5$, obtained by blowing-up $\PP^2$ at $15$ points $p_1, \ldots, p_{15}$ in general position and embedded into $\PP^5$ by the linear system $|H|:=|5h-E_{p_1}-\cdots-E_{p_{15}}|$, where $E_{p_i}$ is the exceptional divisor at the point $p_i\in \PP^2$ and $h\in |\OO_{\PP^2}(1)|$. The White surface is known to projectively Cohen-Macaulay, its ideal being generated by the $3\times 3$-minors of a certain $3\times 5$-matrix of linear forms, see \cite{Gi} Proposition 1.1. In particular, the map $$\mbox{Sym}^2 H^0(X,\OO_X(1))\rightarrow H^0(X,\OO_X(2))$$ is an isomorphism and $X\subseteq \PP^5$ lies on no quadrics. We now let $C\subseteq X$ be a general element of the linear system
$$\Bigl|12h-3(E_{p_1}+\cdots+E_{p_{13}})-2(E_{p_{14}}+E_{p_{15}})\Bigr|.$$
Note that $\mbox{dim } |\OO_X(C)|=6$ and a general element is a smooth curve $C\subseteq \PP^5$ of degree $17$ and genus $14$. Since $H^0\bigl(X, \OO_X(2H-C)\bigr)=0$, it follows that $C$ lies on no quadrics, which finishes the proof.
\end{proof}

\subsection{The Hurwitz space $\H_{19,13}$.} The case $g=19$ is analogous to the situation in genus $14$. We have a morphism $\sigma\colon \hh_{19,13}\rightarrow \mm_{19}$
with generically $5$-dimensional fibres. The Kodaira dimension of both $\mm_{19}$ and of the Hurwitz spaces $\hh_{19,11}$ and $\hh_{19,12}$ is unknown. For a general element
$[C,A]\in \H_{19,13}$, we set $L:=K_C\otimes A^{\vee}\in W^6_{23}(C)$. Observe that $W^7_{23}(C)=\emptyset$, that is, $L$ must be a complete linear series. The multiplication
$$\phi_L\colon \mbox{Sym}^2 H^0(C,L)\rightarrow H^0(C,L^{\otimes 2})$$ is a map between between two vector spaces of the same dimension $28$.  We introduce the degeneracy locus
$$\widetilde{\cD}_{19}:=\Bigl\{[C,A]\in \widetilde{\G}^1_{19,13}:  \mbox{Sym}^2 H^0(C,L)\stackrel{\phi_L}\longrightarrow H^0(C,L^{\otimes 2}) \mbox{ is not an isomorphism}\Bigr\}.$$

\begin{proposition}\label{gen19class}
One has the following formula
$$[\widetilde{\cD}_{19}]=\lambda+\frac{1}{6}[D_3]-\frac{1}{2}[D_0]\in CH^1(\widetilde{\G}^1_{19,13}).$$
\end{proposition}
\begin{proof} The class calculation is very similar to the one in the proof of Proposition \ref{gen14class}. It remains to produce an example of a
pair $(C,L)$, where $C$ is a smooth curve and $L\in W^6_{23}(C)$, for which the map $\phi_L$ is an isomorphism.

We blow-up $\PP^2$ at $21$ general points which we denote by $p, p_1, \ldots, p_{12}$ and $q_1, \ldots, q_8$ respectively. Let $X\subseteq \PP^{12}$ be the surface obtained by embedding $\mbox{Bl}_{21}(\PP^2)$ via the linear system
$|H|:=|6h-E_{p}-E_{p_1}-\cdots-E_{p_{12}}-E_{q_1}-\cdots-E_{q_8}|$. Using again \cite{Gi}, we have that the multiplication map
$$\phi_{\OO_X(1)}\colon \mbox{Sym}^2 H^0(X,\OO_X(1))\rightarrow H^0(X,\OO_X(2))$$ is an isomorphism, therefore  $X\subseteq \PP^6$ lies on no quadrics. We consider $C$ to be a general element of the linear system
$$C\in \Bigl|12h-E_p-4\sum_{i=1}^{12} E_{p_i}-3\sum_{j=1}^{8} E_{q_j}\Bigr|.$$
Then $C$ is a smooth curve of degree $23$ and genus $19$. Since $2H-C$ is not an effective divisor it follows that
$\mbox{Ker}(\phi_{\OO_X(1)})\cong \mbox{Ker}(\phi_{\OO_C(1)})$ and we conclude that the multiplication map $\phi_{\OO_C(1)}$
is an isomorphism as well. This finishes the proof.
\end{proof}

\vskip 4pt

We can now prove the case $g=19$ from Theorem \ref{kodhurw2}. Indeed, combining Proposition \ref{gen19class} and Theorem \ref{canpart2}, we find
$$K_{\widetilde{\G}_{19,13}^1}=[\widetilde{\cD}_{19}]+\sigma^*(7\lambda-\delta_0).$$
Since the class $7\lambda-\delta_0$ is big an $\mm_{19}$, it follows that  $K_{\widetilde{\G}_{19,13}^1}$ is big.

\vskip 4pt

\subsection{The Hurwitz space $\H_{17,11}$.}

The minimal Hurwitz cover of $\mm_{17}$ is $\hh_{17,10}$, but its Kodaira dimension is unknown. We consider the next case
$\sigma\colon \hh_{17,11}\rightarrow \mm_{17}$. As described in the Introduction, a general curve $C\subseteq \PP^6$ of genus $17$ and degree $21$ (whose residual linear system is a pencil $A=K_C\otimes L^{\vee}\in W^1_{11}(C)$) lies on a pencil of quadrics. The general element of this pencil has full rank $7$ and we consider the intersection of the pencil with the discriminant, that is, we require we require that two of the rank $6$ quadrics contained in $\mbox{Ker } \phi_L$ coalesce. We define $\cD_{17}$ to be the locus of pairs
$[C,A]\in \H_{17,11}$ such that this intersection is not reduced.

\begin{theorem}\label{hurwgen17}
The locus $\cD_{17}$ is a divisor and the class of its closure $\widetilde{\cD}_{17}$ in $\widetilde{\G}^1_{17,11}$ is given by
$$[\widetilde{\cD}_{17}]=\frac{1}{6}\Bigl(19\lambda-\frac{9}{2}[D_0]+\frac{5}{6}[D_3]\Bigr)\in CH^1(\widetilde{\G}^1_{17,11}).$$
\end{theorem}
\begin{proof} We are in a position to apply \cite{FR} Theorem  1.2, which deals precisely with degeneracy loci of this type. We obtain
$$[\widetilde{\cD}_{17}]=6\bigl(7c_1(F)-52c_1(\E)\bigr)=6\bigl(39\lambda-7[D_0]+5\gamma\bigr).$$ Using once more Proposition \ref{trans}, we
obtained the claimed formula.

\vskip 3pt

We now establish that $\widetilde{\cD}_{17}$ is a genuine divisor on $\widetilde{\mathcal{G}}^1_{17,11}$. To that end, it suffices to exhibit
\emph{one} smooth curve $C\subseteq \PP^6$ of genus $17$ and degree $21$ such that if $L=\OO_C(1)$, then $\mbox{Ker}(\phi_L)$ is a non-degenerate pencil of quadrics. We proceed by picking $16$ points in $\PP^2$ in general position. We denote these points by $p, p_1, \ldots, p_4$ and $q_1, \ldots, q_{11}$. Let $X\subseteq \PP^6$ be the smooth rational surface obtained by embedding the blow-up of $\PP^2$ at these $16$ points by the very ample linear system
$$\Bigl|7h-3E_p-2\sum_{i=1}^4 E_{p_i}-\sum_{j=1}^{11} E_{q_j}\Bigr|,$$
where $E, E_{p_1}, \ldots, E_{p_{4}}, E_{q_1}, \ldots, E_{q_{11}}$ are the exceptional divisors at the corresponding points in $\PP^2$.

Note that $h^0(X, \OO_X(1))=7$ and $h^0(X, \OO_X(2))=26$. Furthermore, $X$ is projectively normal, therefore $\mbox{Ker}(\phi_{\OO_X(1)})$ is
a pencil of quadrics. It can be directly checked with \emph{Macaulay} that this pencil is non-degenerate (see \cite{FR} Theorem 1.10, including the accompanying \emph{Macaulay} file for how  to carry that out). We now take a general element
$$C\in \Bigl|12h-3E_p-4\sum_{i=1}^4 E_{p_i}-2\sum_{j=1}^{11} E_{q_j}\Bigr|.$$
One checks that $C\subseteq X\subseteq \PP^6$ is a smooth curve of genus $17$ and degree $21$. Furthermore, taking cohomology in the short exact sequence
$$0\longrightarrow \mathcal{I}_{X/\PP^6}(2)\longrightarrow \mathcal{I}_{C/\PP^6}(2)\longrightarrow \mathcal{I}_{C/X}(2)\longrightarrow 0,$$
we conclude that $H^0 \bigl(\PP^6, \mathcal{I}_{X/\PP^6}(2)\bigr)\cong H^0\bigl(\PP^6, \mathcal{I}_{C/\PP^6}(2)\bigr)$, that is, the pencil of quadrics $\mbox{Ker}(\phi_{\OO_C(1)})$ is non-degenerate, which finishes the proof.

\vskip 4pt
\end{proof}

Substituting the expression of $[\widetilde{\cD}_{17}]$ in the formula of the canonical class of the Hurwitz space, we find
$$K_{\widetilde{\G}_{17,11}^1}=\frac{1}{5}[\widetilde{\cD}_{17}]+\frac{3}{5}\sigma^*(7\lambda-\delta_0).$$
Just like in the previous case, since the class $7\lambda-\delta_0$ is big on $\mm_{17}$ and $\lambda$ is ample of $\hh_{17,11}$, Theorem \ref{kodhurw2}
follows for $g=17$ as well.

\vskip 5pt

\subsection{The Hurwitz space $\H_{16,9}$.} This is the most interesting case, for (i) we consider a \emph{minimal} Hurwitz cover
$\sigma:\hh_{16,9}\rightarrow \mm_{16}$ of the $\mm_{16}$ and (ii) $16$ is the smallest genus for which the Kodaira dimension of the moduli space of curves of that genus is unknown, see \cite{Ts}.

We fix a general point $[C,A]\in H_{16,9}$ and,
 set $L:=K_C\otimes A^{\vee}\in W^6_{21}(C)$. It is proven in \cite{F3} Theorem 2.7 that the locus
$\cD_{16}$ classifying pairs $[C,A]$ such that the multiplication map
$$\mu:I_2(L)\otimes H^0(C,L)\rightarrow I_3(L)$$ is not an isomorphism, is a divisor on $\H_{16,9}$.

\vskip 4pt

First we determine the class of the Gieseker-Petri divisor, already mentioned in the introduction.

\begin{proposition}\label{gisp}
One has $[\widetilde{\mathcal{GP}}]=-\lambda+\gamma\in CH^1(\widetilde{\G}^1_{16,9})$.
\end{proposition}
\begin{proof} Recall that we have introduced the sheaves $\cV$ and $\E$ on $\widetilde{\G}^1_{16,19}$ with fibres canonically isomorphic to $H^0(C,A)$ and $H^0(C,\omega_C\otimes A^{\vee})$ over a point $[C,A]\in \widetilde{\G}_{16,9}^1$. We have a natural morphism $\E\otimes \cV\rightarrow f_*(\omega_f)$ and $\widetilde{\mathcal{GP}}$
is the degeneracy locus of this map. Accordingly, we can write
$$[\widetilde{\mathcal{GP}}]=\lambda-2c_1(\E)-8c_1(\cV)=-\lambda+\bigl(\mathfrak{b}-\frac{5}{3}\mathfrak{a}\bigr)=-\lambda+\gamma.$$
\end{proof}

We can now compute the class of the divisor $\widetilde{\cD}_{16}$.
\begin{theorem}\label{class16}
The locus $\widetilde{\cD}_{16}$ is an effective divisor on $\widetilde{\G}_{16,9}^1$ and its class is given by
$$[\widetilde{\cD}_{16}]=\frac{65}{2}\lambda-5[D_0]+\frac{3}{2}[\widetilde{\mathcal{GP}}]\in CH^1(\widetilde{\G}^1_{16,9}).$$
\end{theorem}
\begin{proof} Recall the definition of the vector bundles $\F_2$ and $\F_3$ on $\widetilde{\G}_{16,9}^1$, as well as the expression of their first Chern classes
provided by Proposition \ref{chernhurw1}. We define two further vector bundles $\I_2$ and $\I_3$ on $\widetilde{\G}_{16,9}^1$, via the following exact sequences:
$$0\longrightarrow \I_{\ell} \longrightarrow \mbox{Sym}^{\ell}(\E) \longrightarrow \F_{\ell}\longrightarrow 0,$$
for $\ell=2,3$. Note that $\mbox{rk}(\I_2)=9$, whereas $\mbox{rk}(\I_3)=72$. To make sure that these sequences are exact on the left outside a set of codimension at least $2$
inside $\widetilde{\G}^1_{16,9}$, we invoke \cite{F2}, Propositions 3.9 and 3.10. The divisor $\widetilde{\cD}_{16}$ is then the degeneracy locus of the morphism of vector bundles of the same rank
$$\mu\colon \I_2\otimes \E\rightarrow \I_3,$$
which globalizes the multiplication maps $\mu_{C,L}\colon I_{C,L}(2)\otimes H^0(C,L)\rightarrow I_{C,L}(3)$, where $L=\omega_C\otimes A^{\vee}$ and $[C,A]\in \widetilde{\G}^1_{16,9}$.

\vskip 3pt

Noting that  $c_1\bigl(\mbox{Sym}^3(\E)\bigr)=45c_1(\E)$ and $c_1\bigl(\mbox{Sym}^2(\E)\bigr)=9c_1(\E)$, we compute
$$[\widetilde{\cD}_{16}]=c_1(\I_3)-8c_1(\I_2)-9c_1(\E)=31\lambda-5[D_0]+\frac{3}{2}\gamma.$$
Substituting $\gamma=\lambda+[\widetilde{\mathcal{GP}}]$, we obtain the claimed formula.

\vskip 3pt

It remains to observe that it has already been proved in \cite{F2} Theorem 2.7 that for a general pair $[C,L]$, where $L\in W^6_{21}(C)$, the multiplication
map $\mu_{C,L}$ is an isomorphism.
\end{proof}

The formula (\ref{canrepre}) mentioned in the introduction follows now by using Theorem \ref{class16} and the Riemann-Hurwitz formula for the map $\sigma\colon \widetilde{\G}_{16,9}^1\rightarrow \mm_{16}$. One  writes
$$K_{\widetilde{\G}_{16,9}^1}=13\lambda-2[D_0]+\frac{3}{5}[\widetilde{\mathcal{GP}}]+\frac{2}{5}[\widetilde{\mathcal{GP}}]=\frac{2}{5}[\widetilde{\cD}_{16}]+\frac{3}{5}
[\widetilde{\mathcal{GP}}].$$

\end{document}